\documentclass[11pt]{amsart}
\usepackage{amssymb, amsmath, amsthm}
\usepackage[latin1]{inputenc}
\usepackage{graphicx}
\usepackage[final]{hyperref}
\usepackage{color}
\usepackage{ulem}
\usepackage{epstopdf}

\usepackage[a4paper, centering]{geometry}
\usepackage{color}
\usepackage{graphicx}
\usepackage{scalerel,stackengine}
\usepackage{mathtools}
\usepackage{esint}

\usepackage{mathtools}
\usepackage{enumitem}%  
\usepackage{fancyhdr}%                
\newcommand \D{\mathbb{D}}
\newcommand \CC{\mathbb{C}}
\newcommand \ds{\displaystyle}
\newcommand \deff{\text{def}}
\newcommand \Hl{\mathcal{H}^1}
\newcommand \goto{\underset{\epsilon\rightarrow 0}{\longrightarrow}}

\numberwithin{equation}{section}

\def\R{\mathbb{R}}
\def\C{\mathbb{C}}
\def\N{\mathbb{N}}

\def\C{\mathcal{C}}

\definecolor{verde}{RGB}{20,150,100}

\newcommand{\Om}{\Omega}
\newcommand{\om}{\omega}

\newcommand{\ra}{\rightarrow}

\newcommand{\vps}{\varepsilon}

\newcommand{\rau}{\rightharpoonup}

\newcommand{\sq}{\subseteq}

\newtheorem{thm}{Theorem}[section]
\newtheorem{rem}[thm]{Remark}

\newtheorem{prop}[thm]{Proposition}
\newtheorem{lem}[thm]{Lemma}

\newtheorem{defn}[thm]{Definition}

\begin{document}
\title[]{Stability and instability issues of the Weinstock inequality}

\author[D. Bucur]
{Dorin Bucur}
\address[Dorin Bucur]{  Univ. Savoie Mont Blanc, CNRS, LAMA \\
73000 Chamb\'ery, France
}
\email[D. Bucur]{  dorin.bucur@univ-savoie.fr}
\author[M. Nahon]
{Mickaël Nahon}
\address[Mickaël Nahon]{ Univ. Savoie Mont Blanc, CNRS, LAMA \\
73000 Chamb\'ery, France
}
\email[M. Nahon]{  mickael.nahon@univ-smb.fr}

%\keywords{Robin boundary conditions, minimum principle, boundary behavior}

\date{\today}
\maketitle

\begin{abstract}   
Given two  planar,  conformal, smooth open sets $\Om$ and $\om$,  we prove the existence of a sequence of smooth sets $\Om_n$ which geometrically converges to $\Om$ and such that the (perimeter normalized) Steklov eigenvalues of $\Om_n$ converge to the ones of $\om$. 
As a consequence, we answer a question raised by Girouard and Polterovich on the stability of the Weinstock inequality and prove that the inequality is genuinely unstable. However, under some a priori knowledge of the geometry related to the oscillations of the boundaries, stability may occur.
 \end{abstract}

\section{Introduction}
Let $\Om \sq \R^2$ be a bounded, open, simply connected set with a smooth boundary. We consider the Steklov eigenvalue problem on $\Omega$
\begin{equation*}
\begin{cases}
-\Delta u=0&\text{in }\Om,\\
\frac{\partial u}{\partial n}=\sigma u&\text{on }\partial\Om.
\end{cases}
\end{equation*}
For $k=0,1, 2, \dots $ we denote by $\sigma_k(\Om)$ the $k$-th  eigenvalue defined by  
\[\sigma_k\left(\Omega\right)=\underset{U\in \mathcal{U}_k(\Omega)}{\inf}\ \underset{u\in U\setminus \lbrace 0\rbrace}{\sup}\ \frac{\int\limits_{\Omega}|\nabla u|^2 dx}{\int\limits_{\partial\Omega} u^2d\sigma},\]
where $\mathcal{U}_k(\Omega)$ is  the family   of subspaces of dimension $k+1$ of $H^1(\Omega)$. Then
$$0=\sigma_0(\Om)< \sigma_1(\Om) \le \dots \ra +\infty.$$

Below, we denote by $d_H$ the Hausdorff distance between two sets. The first result of the paper is the following.
\begin{thm}\label{bn01}
Let $\Om, \om \sq \R^2$ be two smooth,  conformal open sets. Then there exists a sequence $(\Om_\epsilon)_{\epsilon>0}$ of smooth open sets  homeomorphic to $\Om$,   with uniformly bounded perimeter such that 
$$d_H(\partial \Om_\epsilon,\partial \Om) \underset{\epsilon\rightarrow 0}{\longrightarrow} 0\quad \mbox{and} \quad  \forall k\in \N,\lim_{\epsilon\rightarrow 0}  |\partial \Om_\epsilon| \sigma_k(\Om_\epsilon)= |\partial \om| \sigma_k(\om).$$
\end{thm}
In other words,  for every $n\in \N$,  in any tubular neighbourhood of $\Om$, one can find a smooth set $\Om_\epsilon$ which has almost the same first (normalized)  $n$ eigenvalues as $\om$. 

This result answers in the negative a question raised by Girouard and Polterovich in \cite{GP17.1} (see also \cite[Open problem 5.21]{GP17.2}) concerning the stability of the Weinstock inequality. Weinstock proved in \cite{WE54} that  the disk maximizes the product between the perimeter and the first Steklov eigenvalue in the class of simply connected planar sets. For the disk, the value of this product equals $2\pi$. The open problem of Girouard and Polterovich reads:
\smallskip

\noindent  {\it  "Let $\Om$ be a planar simply connected domain such that the difference $2\pi- |\partial \Om| \sigma_1(\Om)$ is small. Show that $\Om$ must be close  to a disk (in the sense of Fraenkel asymmetry or some other measure of proximity)."\qquad }

\smallskip

Theorem \ref{bn01} gives a negative answer to the question and, even more, it states that the maximal value $2\pi$ can be asymptotically achieved in the geometric neigbourhood of any smooth simply connected set. 

However, in the recent years, several spectral isoperimetric inequalities have been proved to be stable, in the vein of the quantitative isoperimetric inequality proved by Fusco, Maggi and Pratelli in 2008 \cite{FMP08}.  Stability, involving the Fraenkel asymmetry, holds for the Faber-Krahn, the Saint-Venant or the Sezg\"o-Weinberger inequalites, but also many others. For the Steklov problem, it was proved by Brasco, De Phillipis and Ruffini in 2012 \cite{BDPR12} that the Brock version of the Weinstock inequality is stable. This inequality involves the volume of the set instead of the volume of the boundary as a constraint. Precisely, it is proved in \cite{BDPR12} that in $\R^N$, for every open, bounded and smooth set, it holds
 $$|B|^\frac 1N  \sigma_1(B) - |\Om| \frac 1N \sigma_1(\Om)\ge C_N {\mathcal A}^2(\Om),$$
 where $B$ is a ball and ${\mathcal A}$ is the Fraenkel asymmetry.
 
 Coming back in $\R^2$ to the original Weinstock inequality, if one restricts to the class of convex sets, Weinstock himself implicitly found a stable version of the inequality. Precisely, he proved that if $\Om\sq \R^2$ is a bounded, convex set containing the origin, then
 $$ \pi \int_{\partial \Om} |x|^2 d\sigma- |\Om||\partial \Om|  \ge  \frac{|\partial \Om|}{2}  \int_{S^1} (h -\overline h)^2 d \sigma,$$
 where
 $h$ is the support function of the convex set, and $\overline h$ is its average. This inequality, readily gives a quantitative form of the inequality, in the class of convex sets
\begin{equation}\label{bn03}
2\pi - |\partial \Om| \sigma_1(\Om) \ge \frac {|\partial \Om|}{\int_{\partial \Om} |x|^2 dx}\int_{S^1} (h -\overline h)^2 d \sigma. 
\end{equation}
In a recent paper \cite{GMPT19} (still in the class of convex sets) the right hand side is replaced by
$C {\mathcal A}^\frac 52(\Om)$.
We point out that, very surprisingly,  inequality \eqref{bn03} is published only in the preprint version \cite{W2} and does not figure in the final version the paper \cite{WE54}. 
 
 In $\R^N$, for $N \ge 3$, under a boundary volume constraint, a similar version to the Weinstock inequality is proved to hold in the class of convex sets (see \cite{BFNT}) and it is proved, by Fraser and Schoen, not to hold in the class of contractible domains (see \cite{FS19}). Under convexity hypotheses (see \cite{BFNT}) one has
 \begin{equation}\label{bn02}
  \dfrac{\int_{\partial \Omega} |x|^2 d\sigma }{|\partial \Omega| |\Omega|^\frac 2N}\ge \omega _N^{-\frac 2N},
  \end{equation}
 which by standard use of test functions leads to the $N$-dimensional version of the Weinstock inequality  (above $\omega_N$ is the volume of the unit ball in $\R^N$, and the right hand side corresponds to the ball). 
 A quantitative version   in terms of the Fraenkel asymmetry has been established in \cite{GMPT19}. 
 
Theorem \ref{bn01} in this paper  asserts that there is no general stability of the Weinstock inequality in the class of simply connected sets in $\R^2$. We analyse this issue and give a stability result provided some rigidity is a priori known on the boundary. The convexity assumption in \cite{GMPT19} is such a rigidity. Our result involves some control of the norm of the conformal mapping but allows oscillations of the boundary. Precisely, for all $K>0$, $\alpha\in ]0,1]$, we denote $\mathcal{A}(K,\alpha)$ the set of smooth simply connected sets $\Omega$ with perimeter $2\pi$ such that there exists a conformal mapping $g:\D\rightarrow \Omega$ with:
\[\Vert \log\left(|g'|\right)\Vert_{\C^{0,\alpha}(\partial \D)}\leq K.\]

\begin{thm}\label{bn05}
 There exists a constant $C=C(K, \alpha)>0$ such that for any $\Omega\in \mathcal{A}(K, \alpha)$:
\[\sigma_1(\D)-\sigma_1(\Omega) \geq C \Big (\overline{d_H}(\Omega, \D)\Big )^{2(1+\alpha^{-1})} .\]
\end{thm}
Where $\overline{d_H}(\Omega,\omega)=\inf_{z\in\R^2}d_H(\Omega,\omega+z)$.
We shall prove that this result is sharp in the sense that no weaker a priori estimate is enough to obtain stability; in particular, a bound on $\Vert \log\left(|g'|\right)\Vert_{\C^0(\partial\D)}$ is not sufficient. However, there is no indication that the exponent $2(1+\alpha^{-1})$ is sharp. 

 The paper is organized as follows:   in the second section, we prove Theorem \ref{bn01}. The key of the proof is an approximation result of a weighted Steklov problem by a sequence of Steklov problems without weights but with oscillating boundaries. In the third section, we study the stability of the Steklov eigenvalue on the disk in terms of perturbation of the constant weight on the boundary. The geometric interpretation of this result gives us Theorem \ref{bn05}. The last section is devoted to the study of the sharpness of this result with some explicit computations.

\section{The  Steklov problem and some stability properties}

Throughout the paper we identify $\R^2$ with $\CC$ by $(x_1,x_2) \to z=x_1+ix_2$. For any open Lipschitz domain $\Omega$ of $\R^2$ and any non-negative function $\Theta \in  L^\infty (\partial \Om), \Theta \not=0$
we define  the (generalized) Steklov eigenvalues with the weight $\Theta$ on the boundary by:
\[\sigma_k\left(\Omega,\Theta\right)=\underset{U\in \mathcal{U}_k(\Omega)}{\inf}\ \underset{u\in U\setminus \lbrace 0\rbrace}{\sup}\ \frac{\ds \int\limits_{\Omega}|\nabla u|^2 dx }{\ds \int\limits_{\partial\Omega}\Theta u^2d\sigma},\]
 where $\mathcal{U}_k(\Omega)$ stands for the family of subspaces of dimension $k+1$ in $H^1(\Om)$. The spectrum is well defined, as a direct consequence of the continuous embedding $H^1(\Omega)\hookrightarrow H^{\frac{1}{2}}(\partial\Omega)$ and of the compact embedding $H^{\frac{1}{2}}(\partial\Omega)\hookrightarrow L^2(\partial\Omega)$. 
 
In particular, on a connected bounded Lipschitz domain $\Omega$, we have:
\[\sigma_1(\Omega,\Theta)=\underset{ \int\limits_{\partial\Omega}\Theta u=0}{\underset{u\in H^1(\Omega)}{\inf}}\ \frac{\ds \int\limits_{\Omega}|\nabla u|^2 dx}{\ds \int\limits_{\partial\Omega}\Theta u^2d\sigma}.\]

%If moreover $\Theta$ is smooth enough, we can see the $(\sigma_k(\Omega,\Theta))_k$ as the eigenvalues of the Dirichlet-to-Neumann operator 
%\[H^{\frac{1}{2}}(\partial\Omega)\rightarrow H^{-\frac{1}{2}}(\partial\Omega), \quad u\mapsto \frac{1}{\Theta}\partial_\nu \mathcal{H} u,\]
%where $\mathcal{H}u$ denotes the harmonic extension of $u|_{\partial \Om}$ on $\Omega$ and $\partial_\nu$ is the outward normal derivative.

The instability result lies on two observations:\begin{itemize}[label=\textbullet]
\item Assume $\Omega$ and $\omega$ are two domains with weights $\Theta_\Omega,\Theta_\omega$ on the boundary, we say that $(\Omega,\Theta_\Omega)$ and $(\omega,\Theta_\omega)$ are conformally equivalent if there is a conformal map $g:\Omega\to\omega$ such that $|g'|\Theta_\omega\circ g=\Theta_\Omega$. Under this condition, one has $\sigma_k(\Omega,\Theta_\Omega)=\sigma_k(\omega,\Theta_\omega)$.
\item By an homogeneization process, a domain $\Omega$ with a weight $\Theta\geq 1$ can be approached by domains $(\Omega_\epsilon)$ such that $\sigma_k(\Omega_\epsilon)\underset{\epsilon\rightarrow 0}{\longrightarrow}\sigma_k(\Omega,\Theta)$. Roughly speaking, the sets $(\Omega_\epsilon)$ are locally build as graphs of functions $f_\epsilon$ on $\partial\Omega$, with small $L^\infty$ norm, and such that $\sqrt{1+|f_\epsilon'|^2}\approx \Theta$.
\end{itemize}

This second part is the most technical one.  We start by introducing a notion of convergence of domain that is strong enough for our purpose.

\begin{defn}
We say a sequence of (bounded) domains $(\Omega_{\epsilon})$ converges to a bounded domain $\Omega$ in uniformly if:
\begin{itemize}[label=\textbullet]
\item $\R^2 \setminus \Omega_\epsilon\goto\R^2 \setminus  \Omega$ in the    Hausdorff sense
\item $(\Omega_\epsilon)_{\epsilon>0}$ verify a uniform cone condition, meaning that there exists $\alpha\in ]0,\pi/2[$ and $\delta>0$ such that for any $\epsilon>0$  and any $x\in \partial \Omega_\epsilon$, there exists $\zeta_{x,\epsilon}$ a unitary vector such that for all $y\in \overline{\Omega_\epsilon}\cap B(x,\epsilon)$:
\[\lbrace z\in B(y,\epsilon)\text{ with }\langle z-y,\zeta_{x,\epsilon}\rangle \geq \cos(\delta)|z-y|\rbrace\subset\overline{\Omega_\epsilon}.\]
\end{itemize}
\end{defn}
In particular (see for instance \cite{HP18}), this convergence implies that:
\[\limsup_{\vps \ra 0} |\partial \Omega_\epsilon|<+\infty.\]
There exists a constant $M>0$ such that all $u$ in $H^1(\Omega)$ or $H^1(\Omega_\epsilon)$ can be extended to a function of $H^1(\R^2)$ with
\[\Vert \tilde{u}\Vert_{H^1(\R^2)}\leq M\Vert u\Vert_{H^1(\Omega^\epsilon)\text{ or }H^1(\Omega)}.\]
In addition, there exist a finite number of squares $(C_i)_{i=1,\hdots,K}$ centered in $p_i\in\R^2$ with radius $r_i$, that cover $\partial\Omega$ and all $\partial\Omega_\epsilon$, and such that for any $i$ there exist an orthonormal basis $(e_i,f_i)$ such that
\begin{align*}
C_i&=\lbrace p_i+te_i+sf_i,\ t,s\in [-r_i,r_i]  \rbrace,\\
\partial\Omega\cap C_i&=\lbrace p_i+te_i+g_i(t)f_i,\ t\in [-r_i,r_i]  \rbrace,\\
\partial\Omega_\epsilon\cap C_i&=\lbrace p_i+te_i+g_{i,\epsilon}(t)f_i,\ t\in [-r_i,r_i]  \rbrace,
\end{align*}
where the $(g_i)$ and $(g_{i,\epsilon})$ are uniformly Lipschitz functions. Moreover, the Hausdorff convergence of $(\Omega_\epsilon)$ gives:
\[\underset{1\leq i\leq K}{\sup}\Vert g_i-g_{i,\epsilon}\Vert_{L^\infty}\goto 0.\]

The following proposition extends a result from \cite{Bog17}.
\begin{prop}
Let $\Omega$ be a bounded Lipschitz domain and let $(\Omega_\epsilon)$ be a sequence of domains that converges to $\Omega$ uniformly, with a function $\Theta\in L^\infty(\partial\Omega)$ and a sequence $\Theta_\epsilon\in L^\infty(\partial\Omega_\epsilon)$ such that:
\[\limsup_{\epsilon\rightarrow 0}\Vert\Theta_\epsilon\Vert_{L^\infty(\partial\Omega_\epsilon)}<\infty.\]
We suppose that:
\[\Theta_\epsilon \Hl_{\lfloor \partial\Omega_\epsilon}\rau \Theta\Hl_{\lfloor \partial\Omega}\]
weakly-$*$ in the sense of measures. Let $(u_\epsilon)_\epsilon$ be a sequence of functions  in $H^1(\mathbb{R}^2)$ converging weakly to  $u$. Then
\[\int\limits_{\partial\Omega_\epsilon}\Theta_\epsilon u_\epsilon^2d\sigma\goto \int\limits_{\partial\Omega}\Theta u^2d\sigma.\]
\end{prop}

\begin{proof}

We fix a $i\in\lbrace 1,\hdots,K\rbrace$, and we let $v(t,s)=u(p_i+te_i+sf_i)$ and $v_\epsilon(t,s)=u_\epsilon(p_i+te_i+sf_i)$. We also take $(\psi_i)_i$ a partition of unity associated to the covering $(C_i)$, and 
\begin{align*}
J(t)&=\psi_i(t,g_i(t))\sqrt{1+|g_i'(t)|^2}\Theta(t,g_i(t)),\\
J_\epsilon(t)&=\psi_i(t,g_i(t))\sqrt{1+|g_{i,\epsilon}'(t)|^2}\Theta_\epsilon(t,g_{i,\epsilon}(t)).
\end{align*}
Denoting $I=[-r_i,r_i]$, we only need to show that
\[\int_I v_\epsilon(t,g_{i,\epsilon}(t))^2J_\epsilon(t)dt\goto\int_I v(t,g_i(t))^2J(t)dt.\]
We split the difference in three and use the bound $|J|,|J_\epsilon|\leq M$ for a certain $M>0$ that only depends on the Lipschitz constant of the $(g_{i,\epsilon})$, on $\Vert\Theta\Vert_{L^\infty(\partial\Omega)}$ and on $\sup_{\epsilon}\Vert \Theta_\epsilon\Vert_{L^\infty(\partial\Omega_\epsilon)}$:
\begin{align*}
\left|\int_I ( v_\epsilon(t,g_{i,\epsilon}(t))^2 J_\epsilon(t)-v(t,g_{i}(t))^2 J(t))dt\right|&\leq M\int_I |v_\epsilon(t,g_{i,\epsilon}(t))^2-v_\epsilon(t,g_{i}(t))^2| dt\\
&+M\int_I |v_\epsilon(t,g_{i}(t))^2-v(t,g_{i}(t))^2|dt\\
&+|\int_I v(t,g_{i}(t))^2 (J_\epsilon(t) -J(t))dt|.
\end{align*}
\begin{itemize}[label=\textbullet]
\item \textbf{First term}: We first show that $v_\epsilon(t,g_{i,\epsilon}(t))-v_\epsilon(t,g_{i}(t))\goto 0$ in $L^2(I)$. Since $v_\epsilon\in H^1( \mathbb{R}^2)$, we know by Fubini theorem  that $s\mapsto v_\epsilon(t,s)$ is in $H^1(\mathbb{R})$ for almost every $t\in I$, with derivative $\partial_s v$; thus for almost every $t\in I$, we may write: 
\begin{align*}
|v_\epsilon(t,g_{i,\epsilon}(t))-v_\epsilon(t,g_{i}(t))|&\leq \int\limits_{[g_{i,\epsilon}(t),g_i(t)]}|\partial_s v_\epsilon(t,s)|ds\\
&\leq |g_{i,\epsilon}(t)-g_i(t)|^{1/2}\left( \int\limits_{[g_{i,\epsilon}(t),g_i(t)]}|\nabla v_\epsilon(t,s)|^2 ds\right)^{1/2}.
\end{align*}
We integrate for $t\in I$:
\[\int_I |v_\epsilon(t,g_{i,\epsilon}(t))-v_\epsilon(t,g_{i}(t))|^2 dt \leq \Vert g_i-g_{i,\epsilon}\Vert_{L^\infty(I)}\Vert \nabla u_\epsilon\Vert^2_{L^2(\R^2)}\goto 0.\]
Where we used the fact that $\Vert \nabla v_\epsilon\Vert^2_{L^2(\R^2)}=\Vert \nabla u_\epsilon\Vert^2_{L^2(\R^2)}$ in the last factor. Now, we write:
\begin{align*}
v_\epsilon(t,g_{i,\epsilon}(t))^2-v_\epsilon(t,g_{i}(t))^2=&\Big[v_\epsilon(t,g_{i,\epsilon}(t))-v_\epsilon(t,g_{i}(t))\Big]^2\\
&+2v_\epsilon(t,g_{i}(t))\Big[v_\epsilon(t,g_{i,\epsilon}(t))-v_\epsilon(t,g_{i}(t))\Big].
\end{align*}
And we just use the fact that since $(u_\epsilon)$ is bounded in $H^1(\R^2)$, then it is bounded in $L^2(\partial\Omega)$, so $(v_\epsilon(t,g_i(t))$ is bounded in $L^2(I)$ by a constant $C>0$. We obtain:
\begin{align*}
\int_I |v_\epsilon(t,g_{i,\epsilon}(t))^2-v_\epsilon(t,g_{i}(t))^2| dt&\leq \Vert \nabla u_\epsilon \Vert_{L^2(\mathbb{R}^2)}^2\Vert g_{i,\epsilon}-g_i\Vert_{L^\infty(I)}+2C\Vert \nabla u_\epsilon \Vert_{L^2(\mathbb{R}^2)}\Vert g_{i,\epsilon}-g_i\Vert_{L^\infty(I)}^\frac{1}{2}\\
&\goto 0.
\end{align*}
\item \textbf{Second term}: $u_\epsilon$ converges weakly to $u$ in $H^1(\R^2)$, and so it converges weakly to $u$ in $H^{\frac{1}{2}}(\partial\Omega)$. Since $H^{\frac{1}{2}}(\Omega)\hookrightarrow L^2(\partial\Omega)$ is compact, this means that $u_\epsilon$ converges in the $L^2(\partial\Omega)$ sense to $u$. This implies that $v_\epsilon(t,g_i(t))$ converges to $v(t,g_i(t))$ in $L^2(I)$, which gives that the second term goes to 0.
\item \textbf{Third term}: We use here the hypothesis that $\Theta_\epsilon\Hl_{\lfloor \partial\Omega_\epsilon}\rau  \Theta\Hl_{\lfloor \partial\Omega}$ weakly-$*$ in the sense of measures. This gives that for any continuous function $\varphi\in\C^0(I)$, we have:
\[\int_I \varphi(t)(J_\epsilon(t)-J(t))dt\goto 0.\]
$v(t,g_i(t))^2\in L^1(I)$, so we can approximate it with a continuous function $\varphi$ in the $L^1$ sense. Then:
\begin{align*}
|\int_I v^2 (J_\epsilon-J)dt-\int_I \varphi(J_\epsilon-J)dt|&\leq 2M \int_I|v^2-\varphi|dt.
\end{align*}
Where $M$ is the $L^\infty(I)$ bound of $J_\epsilon$ and $J$. This gives the result.
\end{itemize}
\end{proof}

We are now in position to prove the following homogeneization result.

\begin{prop}
Let $\Omega$ be a bounded Lipschitz domain and let $(\Omega_\epsilon)$ be a sequence of domains that converges to $\Omega$ uniformly, with a function $\Theta\in L^\infty(\partial\Omega)$ and a sequence $\Theta_\epsilon\in L^\infty(\partial\Omega_\epsilon)$ such that:
\[\limsup_{\epsilon\rightarrow 0}\Vert\Theta_\epsilon\Vert_{L^\infty(\partial\Omega_\epsilon)}<\infty.\]
We suppose that:
\[\Theta_\epsilon \Hl_{\lfloor \partial\Omega_\epsilon}\rau \Theta\Hl_{\lfloor \partial\Omega}.\]
weak-$*$ in the measure sense. Then for all $k\geq 1$
\[\sigma_k(\Omega_\epsilon,\Theta_\epsilon)\goto \sigma_k(\Omega,\Theta).\]
\end{prop}

\begin{proof}
We prove the result in two steps:

\noindent{\bf Lower semicontinuity.} For all $\epsilon>0$, let $U_\epsilon\in\mathcal{U}_k(\Omega_\epsilon)$ be a subspace that attains $\sigma_k(\Omega_\epsilon)$, and let $(u_{p,\epsilon})_{p=0,\hdots,k}$ be an adapted basis of it (that is, a basis that is orthonormal relative to the quadratic form $\int_{\partial\Omega_\epsilon}\Theta_\epsilon u^2$ and orthogonal relative to $\int_{\Omega_\epsilon}|\nabla u|^2$).\bigbreak

Since $|\partial\Omega_\epsilon|$ converges, and $|\partial\Omega_\epsilon|\sigma_k(\Omega_\epsilon)$ is bounded according to (Hersch-Payne-Schiffer inequality \cite{HPS75}), we know that the $(u_{p,\epsilon})$ are bounded in $H^1(\R^2)$. Up to an extraction we suppose they converge weakly to some $(u_p)_{p=0,\hdots,k}\in H^1(\R^2)$, in particular we have the following semi-continuity inequality; for all $(a_0,\hdots,a_k)$ with $\sum_p a_p^2=1$:
\[\int\limits_{\Omega}|\nabla \sum_{p}a_p u_p|^2dx\leq\liminf_{\epsilon\rightarrow 0}\int\limits_{\Omega}|\nabla \sum_{p}a_p u_{p,\epsilon}|^2dx\leq \liminf_{\epsilon\rightarrow 0}\sigma_k(\Omega_\epsilon,\Theta_\epsilon ).\]
Moreover, according to the previous theorem, for all $p,p'$, we have:
\[\delta_{p,p'}=\int\limits_{\partial \Omega_\epsilon}\Theta_\epsilon u_{p,\epsilon}u_{p',\epsilon}d\sigma\goto \int\limits_{\partial \Omega}\Theta u_{p}u_{p'} d\sigma.\]
This shows in particular that $(u_p)$ is orthonormal for this scalar product. Thus $\text{Span}(u_0,\hdots,u_p)$ is in $\mathcal{U}_k(\Omega)$ and:
\[\sigma_k(\Omega,\Theta)\leq \underset{\sum_{p}a_p^2=1}{\sup}\ \int\limits_{\Omega}|\nabla\sum_{p}a_p u_p|^2dx \leq \liminf_{\epsilon\rightarrow 0}\sigma_k(\Omega_\epsilon,\Theta_\epsilon ).\]

\noindent{\bf Upper semicontinuity.} Let $V\in\mathcal{U}_k(\Omega)$ be a subspace that attains $\sigma_k(\Omega,\Theta)$, and let $(v_0,\hdots,v_p)$ be an adapted basis of it. When we extend it, $(v_0,\hdots,v_k)$ is a linearly independant family of $H^1(\Omega_\epsilon)$ for small enough $\epsilon$. For all $\epsilon$, let $a_\epsilon=(a_{p,\epsilon})_{p=0,\hdots,k}\in\mathbb{S}^k$ be chosen such that
\[\underset{w\in \text{Span}(v_0,\hdots,v_k)}{\sup}\ \frac{\ds \int\limits_{\Omega_\epsilon}|\nabla w|^2dx}{\ds \int\limits_{\partial\Omega_\epsilon}\Theta_\epsilon w^2d\sigma}=\frac{\ds \int\limits_{\Omega_\epsilon}|\nabla \sum_p a_{p,\epsilon} v_p|^2dx }{\ds \int\limits_{\partial\Omega_\epsilon}\Theta_\epsilon |\sum_p a_{p,\epsilon} v_p|^2d\sigma}.\]
We can suppose each $a_{p,\epsilon}$ converge to a certain $a_p$. Using the previous theorem, we know that:
\[\int\limits_{\partial\Omega_\epsilon}\Theta_\epsilon |\sum_p a_{p,\epsilon} v_p|^2d\sigma\goto \int\limits_{\partial\Omega}\Theta|\sum_p a_{p} v_p|^2d\sigma.\]
And the fact that $\Omega_\epsilon\goto\Omega$ gives:
\[\int\limits_{\Omega_\epsilon}|\nabla \sum_p a_{p,\epsilon} v_p|^2dx\goto \int\limits_{\Omega}|\nabla \sum_p a_{p} v_p|^2dx.\]
Using the two previous limits:
\[\limsup_{\epsilon\rightarrow 0}\sigma_k(\Omega_\epsilon,\Theta_\epsilon )\leq \frac{\ds \int\limits_{\Omega}|\nabla \sum_p a_{p} v_p|^2dx}{\ds \int\limits_{\partial\Omega}\Theta|\sum_p a_{p} v_p|^2d\sigma}\leq\sigma_k(\Omega,\Theta).\]
\end{proof}

\begin{lem}
Let $\Omega,\omega$ be two $\C^{1,\alpha}$ domains of $\CC$ and $g:\Omega\rightarrow \omega$ a conformal map. Let $\Theta\in L^\infty(\partial\omega)$ be a positive function. Then 

\[\sigma_k\left(\omega,\Theta\right)=\sigma_k\left(\Omega,g^*\Theta\right),\]
where $g^*\Theta=|g'|\Theta\circ g$.
\end{lem}
\begin{proof}
Since $\Omega$ and $\omega$ are both $\C^{1,\alpha}$, using the Kellogg-Warschawski theorem (see Theorem 3.6 of \cite{Pom92}) we know that $g$ and its inverse are $\C^{1,\alpha}$ up to the boundary. For any function $u\in H^1(\Omega)$, we let $v=u\circ g^{-1}$. Then a change of variable gives:
\[\int\limits_{\omega}|\nabla v|^2dx =\int\limits_{\Omega}|\nabla u|^2dx\]
and
\[\int\limits_{\partial \omega}\Theta v^2d\sigma=\int\limits_{\partial \Omega}(g^*\Theta) u^2d\sigma.\]
Moreover, if $U$ is in $\mathcal{U}_k(\Omega)$, then the subspace $V=\lbrace u\circ g^{-1},u\in U\rbrace$ is in $\mathcal{U}_k(\omega)$ and the Rayleigh quotient are the same on both space: we have a bijection between $\mathcal{U}_k(\Omega)$ and $\mathcal{U}_k(\omega)$ that preserves the Rayleigh quotient, which proves the result.
\end{proof}

\begin{lem}
Let $\Omega$ be a Lipschitz domain, $\Theta\in L^\infty(\partial\Omega,]1,+\infty[)$. Then there exists a sequence of domains $\Omega_\epsilon$ that converges uniformly to $\Omega$ and such that:
\[\Hl_{\lfloor \partial\Omega_\epsilon}\rau \Theta\Hl_{\lfloor \partial\Omega}\]
weakly-$*$ in the sense of measures.
\end{lem}

\begin{proof}
 We start with the case where $\Omega$ is $\C^2$ and $\Theta\in\C^1(\partial\Omega)$.
We will construct the sequence $(\Omega_\epsilon)$ as a graph on $\partial\Omega$, by adding oscillations on the boundary of $\Omega$ whose amplitude is chosen such that the perimeter of $\Omega$ is locally multiplied by $\Theta$. This is the reason why we need $\Theta$ to be larger than 1.\bigbreak

Let us define the basic oscillation $d:\mathbb{R}\rightarrow\mathbb{R}$ by:
\[d:\begin{cases}x\in [2k,2k+1]\mapsto x-2k\\
x\in [2k+1,2k+2]\mapsto 2k+1-x.\end{cases}\]
Consider $c:\frac{\mathbb{R}}{|\partial\Omega|\mathbb{Z}}\rightarrow \partial\Omega$ a unit-length parametrization of $\partial\Omega$, we define:
\[d_\epsilon:\begin{cases}\partial\Omega\rightarrow \mathbb{R}\\ x\mapsto \epsilon \ d \circ c^{-1}\left(\frac{x}{\epsilon}\right)\end{cases}\]
Notice it is only well defined when $\epsilon=\frac{|\partial\Omega|}{k}$ for a certain $k\in\mathbb{N}^*$.\bigbreak

Finally, let $\lambda:\partial\Omega\rightarrow\mathbb{R}_+^*$ be a smooth function that will be made more precise later. We define our oscillations $(f_\epsilon)$ by:
\[f_\epsilon:\begin{cases}\partial\Omega\rightarrow \mathbb{R}_+\\
x\mapsto \lambda(x)d_\epsilon(x)\end{cases}\]
$f_\epsilon$ verifies $\Vert f_\epsilon\Vert_{L^\infty(\partial\Omega)}=\mathcal{O}\left(\epsilon\right)$ and $f_\epsilon'=\pm \lambda(x)+f_\epsilon(x)\lambda'(x)=\pm\lambda(x)+\mathcal{O}\left(\epsilon\right)$.\bigbreak

We consider the domain $\Omega_\epsilon$ defined with its boundary as
\[\partial\Omega_\epsilon=\lbrace x+f_\epsilon(x)\nu(x),\ x\in\partial\Omega\rbrace.\]
This means that $\Omega_\epsilon$ is defined for any small enough $\epsilon$ as the union of $\Omega$ and of the segments $[x,x+\nu(x)f_\epsilon(x)]$ for $x\in\partial\Omega$. It has a smooth by part boundary, and because it is defined as the graph of a Lipschitz function (where the Lipschitz constant does not depend on $\epsilon$) on the smooth set $\partial\Omega$, the sequence $(\Omega_\epsilon)$ verifies a uniform cone condition.\bigbreak

We now show that $\Hl_{\lfloor \partial\Omega_\epsilon}\rau \Theta\Hl_{\lfloor \partial\Omega}$ in weakly-$*$ in the sense of measures  for a certain choice of $\lambda$. Let $\varphi$ be any continuous function, then
\begin{align*}
\left(\mathcal{H}^1_{\lfloor \partial\Omega_\epsilon}\right)(\varphi)&=\int\limits_{\partial\Omega_\epsilon}\varphi(x)d\sigma\\
&=\int\limits_{\partial\Omega}\varphi(x+\nu(x)f_\epsilon(x))\left(\sqrt{1+\lambda^2}+\mathcal{O}(\epsilon)\right)d\sigma\\
&\goto (\sqrt{1+\lambda^2}\Hl_{\lfloor \partial\Omega})(\varphi),
\end{align*}
where the $\mathcal{O}(\epsilon)$ term contains not only the derivative of $\lambda$, but also the curvature of $\Omega$. By choosing $\lambda=\sqrt{\Theta^2-1}$, we get the result.\bigbreak

Assume  now that $\Omega$ is Lipschitz and $\Theta \in L^\infty(\partial\Omega)$.
 We proceed by regularization to use the previous result: there exists a sequence of $\C^2$ domains $(\Omega_\epsilon)_{\epsilon>0}$ that converges uniformly to $\Omega$, with a sequence $\Theta_\epsilon\in\C^1(\partial\Omega_\epsilon,]1,+\infty[$) that is uniformly bounded in $L^\infty$ and such that:

\[\Theta_\epsilon\Hl_{\lfloor \partial\Omega_\epsilon}\rau \Theta\Hl_{\lfloor \partial\Omega}.\]

For each $\epsilon$, we apply the previous result to get a sequence $\Omega_{\epsilon,\eta}\underset{\eta\rightarrow 0}{\longrightarrow}\Omega_\epsilon$ with:
\[\Hl_{\lfloor \partial\Omega_{\epsilon,\eta}}\underset{\eta\rightarrow 0}{\rau}\Theta_\epsilon\Hl_{\lfloor \partial\Omega_\epsilon}.\]
Using the fact that the space of Radon measures endowed with the weak-$*$ convergence is metrizable, by a diagonal extraction we can find a $\eta(\epsilon)>0$ for each $\epsilon$ such that:
\[\Hl_{\lfloor \partial\Omega_{\epsilon,\eta(\epsilon)}}\rau \Theta\Hl_{\lfloor \partial\Omega}.\]

\end{proof}

\begin{thm}
Let $\Omega$, $\omega$ be two $\C^{1,\alpha}$ domains and $g:\Omega\rightarrow \omega$ be a conformal map between the two. Then there exists a sequence $\Omega_\epsilon$ of domains homeomorphic to $\Omega$ such that $\Omega_\epsilon\goto \Omega$ uniformly and:

\[|\partial \Omega_\epsilon|\sigma_k\left(\Omega_\epsilon\right)\goto|\partial\omega|\sigma_k\left(\omega\right)\]
for any $k$. Moreover, $(|\partial\Omega_\epsilon|)_{\epsilon>0}$ is uniformly bounded.
\end{thm}

\begin{proof}
Let $\Theta(z)=\Lambda|g'(z)|$ for $z\in\partial\Omega$ and $\Lambda>0$ large enough to have $\Theta>1$ everywhere (this is possible because $g'$ does not vanish). Let $\Omega_\epsilon$ be a sequence that converges uniformly to $\Omega$ such that $\sigma_k(\Omega_\epsilon)\goto\sigma_k(\Omega,\Theta)$.\bigbreak
Now, we know:
\[\sigma_k(\Omega,\Theta)=\sigma_k(\omega,\Lambda)=\Lambda^{-1}\sigma_k(\omega).\]
And:
\[|\partial\Omega_\epsilon|\goto\int\limits_{\partial\Omega}\Theta d\sigma=\Lambda|\partial\omega|.\]
This shows the result.
\end{proof}

As a corollary of this result, we obtain the negative answer to the question of the stability of Weinstock's inequality:

\begin{thm}\label{Instab}
Let $\Omega$ be a simply connected domain of $\R^2$ with   $\C^{1,\alpha}$ boundary. There exists a sequence of simply connected domains $(\Omega_\epsilon)_{\epsilon}$ that converges to $\Omega$ uniformly such that:
\[|\partial\Omega_\epsilon|\sigma_1\left(\Omega_\epsilon\right)\goto2\pi\sigma_1(\D).\]
\end{thm}

\section{A stability result under a priori bounds}

In this section, we show that the stability of Weinstock's inequality can still be obtained provided some a priori information on the conformal map $g:\D \rightarrow \Omega$ is known. The family of convex sets, where stability occurs e.g. \cite{GMPT19}, can be described in terms of conformal mappings by the constraint
$$\text{Re}\left(1+z\frac{g''(z)}{g'(z)}\right)\geq 0.$$
 Our a priori condition will only include smooth enough sets but will allow for non-convex and non-starlike domains.

The strategy, based on the relationship between the Steklov problem on $\Omega$ and the weighted Steklov problem on $\D$ with weight $\Theta=|g'|$ on   $\partial \D$  is as follows:
\begin{itemize}[label=\textbullet]
\item We begin by the study of the stability of $\Theta\mapsto \sigma_1(\D,\Theta)$.
From the information that $\left(\ds \int_{\partial\D}\Theta d\sigma\right)\sigma_1(\D,\Theta)$  is close to $|\partial\D|\sigma_1(\D)$ we obtain that $\Theta$ is close to a constant in the $H^{-\frac{1}{2}}(\partial\D)$ norm. This assertion  holds provided that $\Theta$ is normalized to have its center of mass in 0, which amounts to replacing $g$ with $g\circ \phi$ where $\phi$ is a certain conformal automorphism of the disk. 

\item Using the a priori information on $g$, we improve the norm and get that $\Theta$ is close to a constant in   $L^\infty(\partial\D)$. 

\item We transfer this result to sets $\Omega$ by showing that the Hausdorff asymmetry of $\Omega$ is small when $|g'|$ is close to a constant in $L^\infty(\partial\D)$.\end{itemize}
In view of  Theorem  \ref{Instab},   the a priori information on the conformal mapping is crucial.  
 
As previously mentionned, we first begin by the study of the Steklov problem of the disk with weight $\Theta$. For all $\Theta\in L^\infty(\partial \D,\mathbb{R}_+^*)$ it will be practical to introduce the deficit:
\[\deff(\Theta)=\frac{1}{\sigma_1(\partial \D,\Theta)}-1.\]
If $f\in L^2(\partial\D)$, we denote by $\widehat{f}(n)=\frac{1}{2\pi}\int_0^{2\pi}f(e^{it})e^{-int}dt$ its Fourier coefficients, and we define the $H^s$ semi-norm by:
\[\Vert f\Vert_{H^s(\partial \D)}^2=\sum_{n\in\mathbb{Z}^*}|n|^{2s}|\widehat{f}(n)|^2.\]
Provided we restrict ourselves to functions that verify $\widehat{f}(0)=0$, this becomes a norm. Using this definition, it can be checked with Fourier series decomposition that for all $f\in H^{\frac{1}{2}}(\partial\D)$, the harmonic extension $\mathcal{H}f$ of $f$ is well-defined and:
 \[\int_{\D}|\nabla\mathcal{H}f|^2dx =2\pi \Vert f\Vert_{H^{\frac{1}{2}}(\partial\D)}^2.\]

\begin{prop}\label{EstimationSobolev}
Let $\Theta\in L^\infty(\partial \D,\mathbb{R}_+)$ be such that $\widehat{\Theta}(0)=1$ and $\widehat{\Theta}(\pm 1)=0$. Assume morever that $\deff(\Theta)\leq 1$.  Then, for some constant $C>0$ independent of $\Theta$
\[\Vert \Theta -1\Vert_{H^{-\frac{1}{2}}(\partial \D)}\leq C\sqrt{\mbox{\rm \deff}(\Theta)}.\]

\end{prop}

\begin{proof}
Let $u(z)=z$ be the first (complex) eigenfunctions associated to $\Theta\equiv  1$. Let $\phi$ be a normalized and real-valued $H^{\frac{1}{2}}(\partial\D)$ function. Let also $\zeta\in\CC$ be a number that will be fixed later. Suppose also that $\widehat{\phi}(0)=0$; we will see that we lose no generality with this. Then, from the definition of $\deff(\Theta)$:
\[(1+\deff(\Theta))\int_{\D}|\nabla(u+\zeta \mathcal{H}\phi)|^2dx\geq \int_{\partial\D}\Theta|u+\zeta\phi|^2d\sigma-\left|\int_{\partial \D}\Theta(u+\zeta\phi)d\sigma\right|^2.\]
This can be rewritten as:
\[(1+\deff(\Theta))\left(2\pi+2\pi|\zeta|^2 +2 \text{Re}\left(\overline{\zeta} \int_{\partial \D} u\phi d\sigma\right)\right)\geq \]
\[\geq 2\pi+\int_{\partial\D}\Theta\left( 2\text{Re}(\overline{\zeta} u\phi)+|\zeta|^2\phi^2\right)d\sigma-|\zeta|^2\left(\int_{\partial \D}\Theta \phi d\sigma\right)^2.\]
Thus:

\[2\text{Re}\left(\overline{\zeta}\int_{\partial\D}(\Theta-1) u\phi d\sigma\right) \leq\]
\[\leq  2\deff(\Theta)\text{Re}\left(\overline{\zeta}\int_{\partial\D}u\phi d\sigma\right)+ 2\pi\deff(\Theta)+|\zeta|^2\left(2\pi(1+\deff(\Theta))+\left(\int_{\partial \D}\Theta \phi d\sigma\right)^2\right).\]

Using $\deff(\Theta)\leq 1$, this can be simplified to:

\[\text{Re}\left(\frac{\zeta}{|\zeta|}\int_{\partial\D}(\Theta-1) u\phi d\sigma\right)\leq \deff(\Theta)\text{Re}\left(\frac{\overline{\zeta}}{|\zeta|}\int_{\partial\D}u\phi d\sigma\right)+ \frac{\pi \deff(\Theta)}{|\zeta|}+\frac{|\zeta|}{2}\left(4\pi+\left(\int_{\partial \D}\Theta \phi\right)^2 d\sigma\right).\]
Moreover:
\[\int_{\partial \D}\Theta \phi d\sigma=\int_{\partial \D}(\Theta-1) \phi d\sigma\leq 2\pi\Vert \Theta -1\Vert_{H^{-\frac{1}{2}}(\D)},\]
and
\[\left|\int_{\partial\D}u\phi d\sigma\right|=|2\pi\widehat{\phi}(1)|\leq 2\pi.\]
Thus, after normalization by $2\pi$:

\[\text{Re}\left(\frac{\zeta}{|\zeta|}\fint_{\partial\D}(\Theta-1) u\phi d\sigma\right) \leq \deff(\Theta)+\frac{\deff(\Theta)}{2|\zeta|}+\frac{|\zeta|}{2}\left(2+\Vert \Theta -1\Vert_{H^{-\frac{1}{2}}(\D)}^2\right).\]

We optimize in $\zeta$ and get

\[\left|\fint_{\partial\D}(\Theta-1) u\phi d\sigma\right|\leq  \deff(\Theta)+\sqrt{\deff(\Theta)\left(2+\Vert \Theta -1\Vert_{H^{-\frac{1}{2}}(\D)}^2\right)}.\]
Since $\deff(\Theta)\leq \sqrt{\deff(\Theta)}$:

\[\left|\fint_{\partial\D}(\Theta-1) u\phi d\sigma\right|\leq  \left(1+\frac{1}{\sqrt{2}}\right)\sqrt{\deff(\Theta)\left(2+\Vert \Theta -1\Vert_{H^{-\frac{1}{2}}(\D)}^2\right)}.\]

Since $\phi$ is free (but normalized in $H ^{1/2}(\partial \D)$), we get
\[\Vert (\Theta-1)u\Vert_{H ^{-1/2}(\partial \D)}\leq\left(1+\frac{1}{\sqrt{2}}\right)\sqrt{\deff(\Theta)\left(2+\Vert \Theta -1\Vert_{H^{-\frac{1}{2}}(\D)}^2\right)}.\]

By direct computation, using $\widehat{\Theta}(1)=\widehat{\Theta}(-1)=0$:
\[\Vert (\Theta-1)\Vert_{H^{-\frac{1}{2}}(\partial\D)}\leq \sqrt{2}\Vert u(\Theta-1)\Vert_{H^{-\frac{1}{2}}(\partial\D)},\]
so that
\[\frac{\Vert \Theta-1\Vert_{H ^{-1/2}(\partial \D)}}{\sqrt{2+\Vert \Theta -1\Vert_{H^{-\frac{1}{2}}(\D)}^2}}\leq (1+\sqrt{2})\sqrt{\deff(\Theta)},\]
which gives the conclusion.

\end{proof}

In order to get the improved estimate on $\Vert \Theta -1\Vert_{L^\infty(\partial \D)}$, we assume some knowledge on the smoothness of $\Theta$. 

\begin{prop}\label{EstimateCa}
Let $\Theta\in\C^{0,\alpha}(\partial \D,\mathbb{R})$ for $\alpha\in ]0,1]$ be such that $\widehat{\Theta}(0)=1$, $\widehat{\Theta}(\pm 1)=0$, and $\mbox{\rm \deff}(\Theta)\leq 1$. Then:
\[\Vert \Theta -1\Vert_{L^\infty(\partial \D)}\leq C_\alpha(\Vert\Theta\Vert_{\C^{0,\alpha}(\partial \D)})\Vert \Theta -1\Vert_{H^{-\frac{1}{2}}(\partial\D)}^{\frac{1}{1+\alpha^{-1}}}.\]
\end{prop}
\begin{proof}

Let $\phi$ be a positive function, smooth with support in $[-1,+1]$. Let $\phi_\epsilon(e^{it})=\phi(t/\epsilon)$. One can check that:
\[\Vert \phi_\epsilon\Vert_{H^{\frac{1}{2}}(\partial \D)}=\mathcal{O}_{\epsilon\rightarrow 0}(1). \]
Indeed, denoting  $\mathcal{F}$  the Fourier transform in $\R$, we have
\[\Vert \phi_\epsilon\Vert_{H^{\frac{1}{2}}(\partial\D)}^2=\sum_{n\in\mathbb{Z}}|n||\widehat{\phi_\epsilon}(n)|^2=\sum_{n\in\mathbb{Z}}\epsilon|\epsilon n||\mathcal{F}\phi(\epsilon n)|^2\underset{\epsilon\rightarrow 0}{\longrightarrow} \int_{\mathbb{R}}|\xi||\mathcal{F}\phi(\xi)|^2d\xi<\infty.\]
 Suppose now that $\Theta$ reaches its maximum $1+m$ in $e^{it_0}$. We choose $\epsilon$ small enough to have $(\Theta -1)\phi_\epsilon(e^{i(t_0+\cdot)})\geq \frac{1}{2}m\phi_\epsilon(e^{i(t_0+\cdot)})$ : with the regularity on $\Theta$, we can just take $\epsilon=\left(\frac{m}{2[\Theta]_{C^\alpha}}\right)^{\alpha^{-1}}$. Then:
\[\int_{\partial \D}(\Theta-1)\phi_\epsilon(e^{i(t_0+\cdot)})d\sigma \geq  \frac{1}{2}m\int_{\partial\D}\phi_\epsilon(e^{i(t_0+\cdot)})d\sigma=cm^{1+\alpha^{-1}}.\]
For a certain constant $c$ that only depends on $\phi$. This gives $m\leq C\Vert \Theta -1\Vert_{H^{-\frac{1}{2}}(\partial\D)}^{\frac{1}{1+\alpha^{-1}}}$. We can do the same for the minimum of $\Theta-1$ by taking $-\phi_\epsilon$ instead, getting the result.

\end{proof}

\begin{defn}
For all $K>0$, $\alpha\in ]0,1]$, we denote $\mathcal{A}(K,\alpha)$ be the set of simply connected sets $\Omega$ with perimeter $2\pi$ such that there exists a conformal map $g:\D\rightarrow \Omega$ with:
\[\Vert \log\left(|g'|\right)\Vert_{\C^{0,\alpha}(\partial \D)}\leq K.\]
\end{defn}

\begin{rem}\rm
In particular, these domains are  $\C^{1,\alpha}$, so this class excludes domains with angles (e.g. polygonal sets).
\end{rem}

Our main result of stability is the following:

\begin{thm}\label{bmn01}
Let $K>0$, $\alpha\in ]0,1]$. Then there exists a constant $M=M(K,\alpha)>0$ such that for any $\Omega\in \mathcal{A}(K,\alpha)$:

\[\sigma_1(\D)-\sigma_1(\Omega) \geq M \Big (\overline{d_H}(\Omega, \D)\Big )^{2(1+\alpha^{-1})} .\]

Where $d_H$ is the Hausdorff distance to a disk a radius $1$.
\end{thm}
We previously showed that $\Vert \Theta-1\Vert_{L^\infty(\partial\D)}=\Vert |g'|-1\Vert_{L^\infty(\partial\D)}$ is small when $\sigma_1(\Omega)$ is close to $\sigma_1(\D)$. The link between this estimate and the Hausdorff asymmetry of $\Omega$ is given by the following lemma:

\begin{lem}
Let $g$ be a conformal map that sends $\D$ to $\Omega$, then:
\[\overline{d_H}(\Omega, \D)\leq 3\Vert |g'|-1\Vert_{L^\infty(\partial \D)}.\]
\end{lem}
\begin{rem}\rm
The main argument of the proof will be the use of Bloch's theorem that we remind here: {\it let $f$ be any holomorphic function defined on a disk $\D_{z,r}$. Then the image of $f$ contains a disk of radius $L|f'(z)|r$ where $L>0$ is a universal constant called the Landau constant.}
The best known result is $L\geq\frac{1}{2}$, proven in \cite{Ahl38}.
\end{rem}

\begin{proof}
We write $\epsilon=\Vert |g'|-1\Vert_{L^\infty(\partial \D)}$. Using the   maximum principle, we know that for all $z\in \D$:
\[1-\varepsilon\leq |g'|\leq 1+\varepsilon.\]
For all $z\in \D$, the disk $\D_{z,1-|z|}$ is in $\D$, and so, according to Bloch's theorem, the image of $g'$ constains a disk of radius $L|g''(z)|(1-|z|)$ where $L$ is the Landau constant. Since the image of $g'$ contains no disk with radius larger than $\epsilon$, we deduce that for all $z\in D$:
\[|g''(z)|\leq \frac{\varepsilon}{L(1-|z|)}.\]
Now, this implies in particular that for all $z\in \D$:
\begin{align*}
|g(z)-g'(0)z|&\leq \int_0^1 |(1-t)|z|^2 g''(tz)|dt\\
&\leq \int_0^1 (1-t)|z|^2 \frac{\epsilon}{L(1-t|z|)}dt\\
&\leq\frac{\epsilon}{L}.
\end{align*}
In particular;
\[\D_{1-(1+L^{-1})\epsilon}\subset g(\D)\subset \D_{1+(1+L^{-1})\epsilon}.\]
Since $L\geq\frac{1}{2}$, we get the result.
\end{proof}
We can now prove the main result.
\begin{proof}  (of Theorem \ref{bmn01})
The set $\Omega$ belongs to  $\mathcal{A}(K,\alpha)$, so there exists a conformal representation $g$ of $\Omega$ for which $\Vert \log\left(|g'|\right)\Vert_{C^{0,\alpha}}\leq K$. We let $\Theta=|g'|_{|\partial \D}$. The perimeter constraint gives $\widehat{\Theta}(0)=1$ and we suppose for now that $\widehat{\Theta}(\pm 1)=0$.\bigbreak

The hypothesis $\Omega\in \mathcal{A}(K,\alpha)$ tells us that $\Theta$ is bounded by a constant depending only on $K$ in $\C^{0,\alpha}$. By applying Proposition \ref{EstimateCa}, we know that there exists a constant $M=M(K)>0$ such that:
\[\Vert |g'|-1\Vert_{L^\infty(\partial \D)}\leq M|\sigma_1(\D)-\sigma_1(\Omega)|^{\frac{1}{2(1+\alpha^{-1})}}.\]
According to the previous lemma, we get:
\[\overline{d_H}(\Omega, \D)\leq 3 M|\sigma_1(\D)-\sigma_1(\Omega)|^{\frac{1}{2(1+\alpha^{-1})}}.\]

Here we supposed that $g$ was normalized in the sense that $\widehat{\Theta}(\pm 1)=0$. Let us show that we can always normalize $g$ without changing $\Vert \log(|g'|)\Vert_{\C^{\alpha}(\partial \D)}$ too much: the idea is that there exists a unique $\zeta_g\in D$ such that $g\circ \phi_{\zeta_g}^{-1}$ is well-normalized, where $\phi_\zeta(z):=\frac{z+\zeta}{1+\overline{\zeta}z}$ is a conformal automorphism of the disk. It could happen that the $\C^{0,\alpha}$ norm of $g\circ \phi_{\zeta_g}^{-1}$ explodes when $|\zeta_g|$ gets too close to $1$. The following lemma shows that it is not the case:
\begin{lem} Let $g$ be a conformal map, then there exists $r=r\left(\Vert\log\left(|g'|\right)\Vert_{L^\infty(\D)}\right)$ in $[0,1[$ such that $|\zeta_g|\leq r$.
\end{lem}
\begin{proof}
We let $\Theta=|g'_{|\partial \D}|$, and $K=\Vert\log\left(|g'|\right)\Vert_{L^\infty(\D)}$. This means that $e^{-K}\leq \Theta\leq e^K$. We are interested in the quantity:
\begin{align*}
F(\zeta)&=\int_{\partial \D}|(g\circ \phi_{-\zeta})'(z)|zd\sigma\\
&=\int_{\partial \D}\Theta(\phi_{-\zeta}(z))\frac{1-|\zeta|^2}{|1-\overline{\zeta}z|^2}zd\sigma.
\end{align*}
To estimate $|\zeta_g|$, we use the following topological criteria: if $\langle F(\zeta),\zeta\rangle >0$ for all $\zeta\in \partial \D_r$, then $|\zeta_g|\leq r$. Let $r>0$, we estimate:
\begin{align*}
\frac{1}{1-r^2}\langle F(re^{it_0}),e^{it_0}\rangle &=\int_{\partial \D}\Theta(\phi_{-re^{it_0}}(z))\frac{\langle z,e^{it_0}\rangle}{|1-rz|^2}d\sigma\\
&= \int_{0}^{2\pi}\Theta(\phi_{-re^{it_0}}(e^{i(t_0+t)}))\frac{\cos(t)}{|1-re^{it}|^2}dt\\
&\geq  \int_{|t|\leq \pi/2}e^{-K}\frac{\cos(t)}{|1-re^{it}|^2}dt+ \int_{|t|\geq \pi/2}e^{K}\frac{\cos(t)}{|1-re^{it}|^2}dt\\
&\geq e^{-K}\int_{|t|\leq \pi/2}\frac{\cos(t)}{|1-re^{it}|^2}dt-\frac{\pi}{2}e^{K}\\
&\underset{r\rightarrow 1}{\longrightarrow}+\infty\text{ (uniformly in }t_0).
\end{align*}
This proves that for a certain $r\in ]0,1[$ that only depends on $K$, we have $|\zeta_g|\leq r$.
\end{proof}
\end{proof}

\section{Further remarks}

The study of the stability of $\Theta\mapsto\sigma_1(\D,\Theta)$ is independent of the fact that $\Theta=|g'|$ for a certain conformal map $g$. One could wonder if there is an equivalence between the study of the Steklov problem on domains $\Omega$ and the study of the weighted Steklov operator on the disk. We show here that it is the case provided we allow domains that may overlap, meaning they are immerged in the plane and are seen as images of holomorphic functions with non-vanishing derivative that are not necessarily globally injective. We also show that in our stability result, the domains cannot overlap when the first Steklov eigenvalue is close to that of the disk.

\begin{prop}
Let $\Theta\in\C^0(\partial \D,\mathbb{R}_+^*)$, then there exists an holomorphic function with non-vanishing derivative $g:\D\rightarrow \CC$ such that $|g_{|\partial \D}'|=\Theta$. Moreover, it is unique up to an affine isometry.
\end{prop}
\begin{proof}
Let $u$ be the harmonic extension of $\log(\Theta)$ on $\D$. Let $v$ be a harmonic conjugate of $u$ (it can be unique if we fix $v(0)=0$). Then take $g$ as an integral of $e^{u+iv}$ (it exists because $\D$ is simply connected). Then $|g'|=e^u$, which is equal to $\Theta$ on $\partial\D$.\medskip

For the uniqueness, consider $f$ and $g$ two such functions, then $\log(|f'|)$ and $\log(|g'|)$ are two harmonic functions that coincide on $\partial\D$, so they are the same. Then $\text{Im}(\log(f'))$ and $\text{Im}(\log(g'))$ (which are defined up to a constant in $2\pi\mathbb{Z}$) are conjugated to the same harmonic functions, so they differ by a constant.
\end{proof}

Here is a criterium to get a domain that does not overlap, meaning that the function $g$ defined above is injective:
\begin{prop}\label{overlap}
Let $\Theta\in \C^0(\partial  \D,\mathbb{R}_+^*)$. If $\Vert \Theta -1 \Vert_{L^\infty(\partial\D)}\leq \frac{1}{5}$, then $\Theta$ defines a domain that does not overlap.
\end{prop}
\begin{rem}
In particular, for all $\alpha\in ]0,1[$, $K>0$, there exists $\delta>0$ such that if $g:\D\rightarrow \CC$ is a holomorphic function with non-vanishing derivative that verifies:
\begin{itemize}[label=\textbullet]
\item $\int_{\partial \D}|g'|d\sigma=2\pi$,
\item $\Vert \log\left(|g'|\right)\Vert_{\C^{0,\alpha}(\D)}\leq K$,
\item $\sigma_1(\D,|g'_{|\partial\D}|)\geq 1-\delta$.
\end{itemize}
Then $g$ is injective and defines a domains that does not overlap. Indeed, Proposition \ref{EstimateCa} shows that $\Vert |g'|-1\Vert_{L^\infty(\partial\D)}\leq C(K,\alpha)\delta^{2(1+\alpha^{-1})}$, which is less than $\frac{1}{5}$ when $\delta$ is small enough.
\end{rem}
\begin{proof}
Write $\Vert \Theta -1 \Vert_{L^\infty(\partial\D)}=\epsilon$, $g$ a conformal map such that $|g'_{|\partial\D}|=\Theta$. We want to show that $g$ is univalent. We use Theorem 1.11 of \cite{Pom92}; it is enough to show that for all $z\in \D$:
\[(1-|z|^2)\left|z\frac{g''(z)}{g'(z)}\right|\leq 1.\]
Since $||g'|-1|\leq \epsilon$, it means that $g'(\D)$ contains no disk with larger radius than $\epsilon$, and so, using Bloch's theorem:
\[L(1-|z|)|g''(z)|\leq \epsilon.\]
Thus we can verify the univalence criteria:
\[(1-|z|^2)\left|z\frac{g''(z)}{g'(z)}\right|\leq (1-|z|^2)\frac{\epsilon}{L(1-|z|)(1-\epsilon)}\leq\frac{2\epsilon}{L(1-\epsilon)}.\]
This is less than $1$ when $\epsilon\leq \frac{L}{2+L}$. Since it has been show that $L\geq \frac{1}{2}$, the above is true as soon as $\epsilon\leq \frac{1}{5}$.
\end{proof}
\medskip
\noindent{\bf Sharpness of the stability result.} One could wonder if an a priori bound on $\Vert \log\left(|g'|\right)\Vert_{L^\infty(\partial\D)}$ is enough to obtain stability. We prove below that this is not the case.

\begin{prop}
Let $\Omega$ be the image of a smooth conformal map $g\in\C^1(\overline{\D})$ such that
\[\max_{\partial\D}|g'|< \frac{4}{\pi}\min_{\partial\D}|g'|.\]
Then there exists a sequence of domains $\Omega_n=g_n(\D)$ with $g_n\in \C^1(\overline{\D})$ such that 
\begin{align*}
 &|\partial\Omega_n|\sigma_k(\Omega_n)\underset{n\rightarrow\infty}{\longrightarrow} |\partial\D|\sigma_k(\D),\text{ for all }k\geq 0,\\
 &\Omega_n\underset{n\rightarrow\infty}{\longrightarrow}\Omega\text{ uniformly},\\
 &\sup_{n\in\mathbb{N}}\Vert \log |g_n'|\Vert_{\C^0(\D)}<\infty.\\
\end{align*}
\end{prop}

\begin{proof}
Consider $f$ an holomorphic function defined on $\D$ that never takes the value $0$ and such that $|f|<|g'|$ on $\partial\D$. Consider:
\[g_n(z)=g(z)+\frac{z^{n+1}}{n+1}f(z).\]
We show that, at least for $n$ large enough, this defines a sequence of domains that converges uniformly to $\Omega$; we need to check the uniform cone condition. Since $\Omega$ is $\C^1$, there exists a finite number of squares that cover $\partial\Omega$, called $(C_i)$ and oriented by the orthonormal basis $(e_i,f_i)$ such that:
\begin{align*}
C_i&=\lbrace p_i+te_i+sf_i,\ t,s\in [-r_i,r_i]  \rbrace,\\
\partial\Omega\cap C_i&=\lbrace p_i+te_i+h_i(t)f_i,\ t\in [-r_i,r_i]  \rbrace.\\
\end{align*}
Where $\Vert h_i\Vert_{\C^1(\partial\D)}\leq \delta$ for a certain $\delta$ that can be chosen arbitrarily small. We also write $c(t)=g(e^{it})$ and $c_n(t)=g_n(e^{it})$, as well as $I_i=c^{-1}(C_i)$. Then, for all $t\in I_i$ we have
\[\left|\frac{\langle c',f_i\rangle}{\langle c',e_i\rangle}\right|\leq \delta.\]
For a large enough $n$, there exists $\eta>0$ such that $|(c_n-c)'|<(1-\eta)|c'|$, because of the condition that $|f|<|g'|$ (any $\eta$ smaller than $\inf\frac{|g'|-|f|}{|g'|}$ works for $n$ large enough). In particular, $|c_n'|\leq 2|c'|$ and:
\begin{align*}
\left|\langle c_n',e_i\rangle\right|&\geq\left|\langle c',e_i\rangle\right|-\left|c_n'-c'\right|\\
&\geq\left|\langle c',e_i\rangle\right|-(1-\eta)|c'|\\
&\geq\left|\langle c',e_i\rangle\right|-(1-\eta)|\langle c',e_i\rangle|-(1-\eta)|\langle c',f_i\rangle|\\
&= \left(\eta-(1-\eta)\delta\right) \left|\langle c',e_i\rangle\right|.\\
\end{align*}
We choose $\delta\leq\frac{\eta}{2(1-\eta)}$; with this:
\[\left|\frac{\langle c_n',f_i\rangle}{\langle c_n',e_i\rangle}\right|\leq \frac{4}{\eta}\frac{|c'|}{|\langle c',e_i\rangle|}.\]
And so, letting $M=\max_{i}\left(\frac{4}{\eta}\max_{I_i}\frac{|c'|}{|\langle c',e_i\rangle|}\right)$, there exists a sequence of $M$-Lipschitz functions $(h_{i,n})$ such that for all $i$:
\[\partial\Omega_n \cap C_i=\lbrace p_i+te_i+h_{i,n}(t)f_i,\ t\in [-r_i,r_i]  \rbrace.\]
Which proves that the sequence $(\Omega_n)$ verifies a uniform cone condition.

\begin{lem}
 Under these circumstances:
\[\Hl_{\lfloor \partial \Omega_n}\underset{n\rightarrow\infty}{\rau}\Theta\Hl_{\lfloor \partial \Omega}\text{ weakly-$*$ in the sense of measures},\]
where $\Theta=P\left(\left|\frac{f\circ g^{-1}}{g'\circ g^{-1}}\right|\right)$, with:
\[P(a):=\fint_0^{2\pi}\sqrt{1+a^2+2a\cos(t)}dt.\]
Moreover, $P$ is a strictly increasing function with $P(0)=1$ and $P(1)=\frac{4}{\pi}$.
\end{lem}
\begin{proof}
We already know that $g_n$ converges uniformly to $g$, and so it is easier to prove the lemma on the pullback of these measures through $g$ and $g_n$ on $\partial \D$. Let $\varphi\in\C^0(\partial\D,\R)$, we search for the limit of $\int_{\partial\D}\varphi |g_n'|d\sigma$. We compute:
\begin{align*}
\int_{\partial\D}\varphi |g_n'|d\sigma&=\int_{\partial\D}\varphi |g'+z^n f+\mathcal{O}(1/n)|d\sigma\\
&=\int_{\partial\D}\varphi |g'+z^n f|d\sigma+\mathcal{O}(1/n)\\
&=\int_{\partial\D}\varphi |g'|\left|1+z^n \frac{f}{g'}\right|d\sigma+\mathcal{O}(1/n)\\
&\underset{n\rightarrow\infty}{\longrightarrow}\int_{\partial\D}\varphi |g'|P\left(\left|\frac{f}{g'}\right|\right)d\sigma.
\end{align*}
Where we used the fact that for any continuous and $2\pi$-periodic function $\psi$, $\psi(n\cdot)$ converges weakly-* to $\fint \psi$.
\end{proof}

This means that, in order to have $|\partial\Omega_n|\sigma_1(\Omega_n)\rightarrow 2\pi$, we need to chose $f$ such that $\left(\ds \int_{\partial\Omega}\Theta d\sigma\right)\sigma_1(\Omega,\Theta)=\left(\ds \int_{\partial \D}g^*\Theta d\sigma\right)\sigma_1(\D,g^*\Theta)$ is equal to $2\pi$, which is the case of equality in Weinstock's inequality. In other words, we need a constant $\Lambda>0$ such that $g^*\Theta=\Lambda$, which is equivalent to:
\[P\left(\left|\frac{f}{g'}\right|\right)=\frac{\Lambda}{|g'|}.\]
This is only possible when $\frac{\Lambda}{|g'|}$ takes value in $]1,\frac{4}{\pi}[$, and so there exists such a constant $\Lambda$ if and only if $\max_{\partial\D}|g'|< \frac{4}{\pi}\min_{\partial\D}|g'|$. Fix such a $\Lambda$. Then the previous relation can be written:
\[|f|=|g'|P^{-1}\left(\frac{\Lambda}{|g'|}\right).\]
This incites us to define $f$ as $f=e^{u+iv}$, where $u=\mathcal{H}\log\left(|g'|P^{-1}\left(\frac{\Lambda}{|g'|}\right)\right)$ and $v$ is an harmonic conjugate of $u$. With this definition, since $0<P^{-1}<1$, we get that $0<|f|<|g'|$ on $\partial\D$, and with the maximum principle we get $0<|f|<|g'|$ on $\D$. With $f$ defined as such, we obtain that for all $k\geq 0$:
 \[|\partial\Omega_n|\sigma_k(\Omega_n)\underset{n\rightarrow\infty}{\longrightarrow}\left(\int_{\partial\Omega}\Theta d\sigma \right)\sigma_k(\Omega,\Theta)=\left(\int_{\partial \D}\Lambda d\sigma \right)\sigma_k(\D,\Lambda)=|\partial\D|\sigma_k(\D).\]
\end{proof}

\begin{rem}
A consequence of this instability result is the following: there is no inequality of the form:
\[\overline{d_H}(g(\D),\D)\leq \epsilon\left(\Vert |g'|-1\Vert_{H^{-\frac{1}{2}}(\partial\D)}^{\lambda}\Vert |g'|-1\Vert_{L^\infty(\partial\D)}^{1-\lambda}\right),\]
where $\lambda\in ]0,1]$, $\epsilon(t)\underset{t\rightarrow 0}{\longrightarrow} 0$, and $g$ is any conformal map. Indeed, such an inequality would lead to a stability result with a priori bound in $L^\infty$, which was shown to fail.
\end{rem}

\noindent{\bf Lower bound on the exponent.}
We now compute the spectrum on the special case where $\Theta-1$ is a sine function: this help us prove that the optimal exponent of stability is no less than $2$, whereas the exponent we obtain in our result is $2(1+\alpha^{-1})$.

\begin{lem}\label{EstimateAbove}
Let $\alpha\in ]0,1[$, $N\geq 4$, and $\Theta\left(e^{it}\right)=1+\alpha\cos(Nt)$. Then we have the estimate:
\[\frac{\alpha^2}{CN}\leq 
\deff(\Theta)\leq \frac{\alpha^2}{N-3},\]
where $C$ is a constant that does not depend on $N$ or $\alpha$.
\end{lem}

\begin{proof}
We begin with the second inequality. Let $u\in\C^\infty(\partial\D)$ be a finite sum of trigonometric functions, we call $c_n$ its Fourier serie (it has finite support), and $(\Theta_n)$ the fourier serie of $\Theta$. Then the harmonic extension of $u$ is given by
\[\mathcal{H}u(re^{it})=\sum_{n\in\mathbb{Z}}c_n r^{|n|}e^{int}\]
We can compute explicitly
\begin{align*}
\int\limits_{D}|\nabla\mathcal{H}u|^2dx &=2\pi\sum_{n\in\mathbb{Z}^{*}}|nc_n^2|,\\
\int\limits_{\partial\D}\Theta u^2d\sigma&=2\pi\sum_{k,l\in\mathbb{Z}}c_k\overline{c_l}\Theta_{l-k},\\
\int\limits_{\partial\D}\Theta ud\sigma =&2\pi\sum_{n\in\mathbb{Z}}c_n\overline{\Theta_k}.
\end{align*}
We are only interested in test functions such that $\int\limits_{\partial\D}\Theta u=0$: we will suppose from now on that
\[c_0=-\sum_{n\in\mathbb{Z}^*}c_n\overline{\Theta_k}.\]
$\deff(\Theta)$ verifies, for all sequence $(c_n)_{n\in\mathbb{Z}}$ that verify the above condition on $c_0$, that:
\[(1+\deff(\Theta))\sum_{n\in\mathbb{Z}^{*}}|nc_n^2|\geq \sum_{k,l\in\mathbb{Z}}c_k\overline{c_l}\Theta_{l-k}.\]
The value $c_0$ appears in the right-hand side, and we may replace it with $-\sum_{n\in\mathbb{Z}^*}c_n\overline{\Theta_k}$. We get that for all sequence $(c_n)_{n\in\mathbb{Z}^*}$:
\[(1+\deff(\Theta))\sum_{n\in\mathbb{Z}^{*}}|nc_n^2|\geq \sum_{k,l\in\mathbb{Z}^*}c_k\overline{c_l}\Theta_{l-k}-|\sum_{n\in\mathbb{Z}^*}c_n\overline{\Theta_n}|^2.\]
From this we want information on $(\Theta_n)$;  if we find a number $\delta$ such that for all finitely supported sequence $(c_k)_{k\neq 0}$, we have:
\[ \sum_{k,l\in\mathbb{Z}^*}c_k\overline{c_l}\Theta_{l-k}\leq (1+\delta)\sum_{n\in\mathbb{Z}^{*}}|nc_n^2|,\]
then $\deff(\Theta)\leq\delta$.\bigbreak

Here $\Theta_n=\begin{cases}1\text{ for }n=0\\ \alpha/2 \text{ for }n=\pm N.\\ 0 \text{ elsewhere }\end{cases}$\\
We write, for a parameter $\epsilon>0$ that we will choose later:
\begin{align*}
\sum_{k,l\in\mathbb{Z}^*}c_k\overline{c_l}\Theta_{l-k}&=\sum_{n\in\mathbb{Z}^*}|c_n|^2+\frac{\alpha}{2}\sum_{k,l\in\mathbb{Z}^*,\ |k-l|=N}c_k\overline{c_l}\\
&\leq \sum_{n\in\mathbb{Z}^*}|c_n|^2+\alpha\sum_{k,l\in\mathbb{Z}^*,\ |k-l|=N,\ |k|\leq N/2}|c_k c_l|+\frac \alpha 2 \sum_{k,l\in\mathbb{Z}^*,\ |k-l|=N,\ |k|,|l|> N/2}|c_k c_l|\\
&\leq \sum_{n\in\mathbb{Z}^*}|c_n|^2+\alpha\sum_{k,l\in\mathbb{Z}^*,\ |k-l|=N,\ |k|\leq N/2}\frac{\epsilon}{2}|c_k|^2+\frac{1}{2\epsilon}|c_l|^2\\
&+\frac \alpha 2 \sum_{k,l\in\mathbb{Z}^*,\ |k-l|=N,\ |k|,|l|> N/2}\frac{1}{2}|c_k|^2+\frac{1}{2}|c_l|^2\\
&\leq\sum_{n\in\mathbb{Z}^*,\ |n|\leq N/2}\left(1+\alpha\epsilon\right)|c_n|^2+\sum_{n\in\mathbb{Z}^*,\ |n|> N/2}\left(1+\alpha+\frac{\alpha}{2\epsilon}\right)|c_n|^2\\
&\leq\sum_{n\in\mathbb{Z}^*,\ |n|\leq N/2}\left(1+\alpha\epsilon\right)|c_n|^2+\sum_{n\in\mathbb{Z}^*,\ |n|> N/2}\left(2+\frac{\alpha}{2\epsilon}\right)|c_n|^2.\\
\end{align*}

We choose $\epsilon$ such that $\left(2+\frac{\alpha}{2\epsilon}\right)=\frac{N+1}{2}$; we take $\epsilon=\frac{\alpha}{N-3}$. Thus:
\begin{align*}
\sum_{k,l\in\mathbb{Z}^*}c_k\overline{c_l}\Theta_{l-k}&\leq \sum_{n\in\mathbb{Z}^*,\ |n|\leq N/2}\left(1+\frac{\alpha^2}{N-3}\right)|c_n|^2+\sum_{n\in\mathbb{Z}^*,\ |n|> N/2}|n c_n^2|\\
&\leq \left(1+\frac{\alpha^2}{N-3}\right)\sum_{n\in\mathbb{Z}^{*}}|nc_n^2|.
\end{align*}
This gives the second inequality.\bigbreak

For $N\geq 4$, we get $\deff(\Theta)\leq 1$, and so we can apply the result \ref{EstimationSobolev} to $\Theta$. This gives:
\[|\widehat{\Theta}(N)|\leq C\sqrt{N\deff(\Theta)},\]
for a certain constant $C>0$, and so:
\[\frac \alpha 2 \leq C\sqrt{N\deff(\Theta)}.\]
which is equivalent to the first inequality, up to replacing $C$ with $2C$.\medskip

\end{proof}

\begin{prop}
Let $\epsilon>0$, $\alpha>0$, there exists a sequence of domains $\Omega_n$ such that 
\begin{align*}
 &|\partial\Omega_n|=2\pi\text{ for all }n,\\
 &\sup_{n\in\mathbb{N}}\Vert \log |g_n'|\Vert_{\C^{\alpha}}<\infty,\\
 &\sup_{n\geq 0} \frac{\sigma_1(\D)-\sigma_1(\Omega_n)}{ \Big (\overline{d_H}(\Omega_n, \D)\Big )^{2-\epsilon}}<\infty.
\end{align*}
\end{prop}
\begin{rem}
Here $\alpha$ can be arbitrarily large, which means we can ask for a much stronger a priori bound and the optimal exponent is still bigger than $2$.
\end{rem}
\begin{proof}
Consider the sequence $\Omega_n=g_n(\D)$ defined by the weight $\Theta_n(t)=1+a_n\cos(nt)$, for a sequence $a_n$ that will go to $0$. Since $\Vert \Theta_n-1\Vert_{L^\infty(\partial\D)}$ is less than $\frac{1}{5}$ for $n$ big enough, we know that this defines a domain that does not overlap. The above estimate gives us that:
\[\deff(\Theta_n)\leq \frac{a_n^2}{n-3}\leq C\frac{a_n^2}{n}.\]
For a constant $C>0$. Let $h_n$ be defined by $h_n=\log\left(|g_n'|\right)$: it is the unique harmonic function verifying $h_n=\log(\Theta_n)$ on the boundary. We develop it as:
\begin{align*}
h_n(e^{it})=&\log\left(1+a_n\cos(nt)\right)\\
=&\sum_{k\geq 1}\frac{(-1)^{k-1}}{k}a_n^{k}\cos(nt)^k\\
=&\sum_{k\geq 1}\left(\frac{(-1)^{k-1}}{k}a_n^{k}\sum_{p=0}^{\lfloor k/2\rfloor}\frac{1}{2^k}\binom{k}{p}(2\cos(n(2k-p)t)1_{p\neq k/2}+1_{p=k/2})\right).
\end{align*}
In particular, since $h_n$ is harmonic, its expression in the disk is:
\[h_n(re^{it})=\sum_{k\geq 1}\left(\frac{(-1)^{k-1}}{k}a_n^k\sum_{p=0}^{\lfloor k/2\rfloor}\frac{1}{2^k}\binom{k}{p}(2r^{n(2k-p)}\cos(n(2k-p)t)1_{p\neq k/2}+1_{p=k/2})\right).\]
We can verify the a priori hypothesis on $h_n=\log\left(|g_n'|\right)$; there is a constant $C_\alpha>0$ such that:

\[\left[\log\left(|g_n'|\right)\right]_{\C^{0,\alpha}(\partial\D)}\leq C_\alpha\sum_{k\geq 1}(nk)^\alpha a_n^k,\]
which is bounded as soon as $n^{\alpha}a_n$ is bounded: this will be verified later. $h_n$ is the real part of the holomorphic function: 
\[\log(g_n'(z))=ib_n+\sum_{k\geq 1}\left(\frac{(-1)^{k-1}}{k}a_n^k\sum_{p=0}^{\lfloor k/2\rfloor}\frac{1}{2^k}\binom{k}{p}(2z^{2k-p}1_{p\neq k/2}+1_{p=k/2})\right),\]
for a certain branch of the logarithm and $b_n\in\mathbb{R}$. In particular, for a certain constant $C>0$:
\[|g_n(z)-ib_n-a_n z^n|\leq \sum_{k\geq 2}\frac{1}{k}a_n^k\leq Ca_n^2.\]
This means that $g_n'(z)=e^{\log(g_n'(z))}=e^{ib_n}\left(1+a_n z^n+k_n(z)\right)$ where $k_n$ is an holomorphic function that verifies $|k_n(z)|\leq Ca_n^2$ for a certain constant $C$.

We lose no generality in supposing that $b_n=0$ for all $n$. Suppose now that $n$ is odd, $d_H(\Omega_n,\D)$ can be estimated from below by:
\[g_n(1)=\int_0^1 g_n'(r)dr=1+\frac{a_n}{n}+\mathcal{O}\left(\frac{a_n^2}{n}\right),\]
\[g_n(-1)=\int_0^1 g_n'(r)dr=-1-\frac{a_n}{n}+\mathcal{O}\left(\frac{a_n^2}{n}\right).\]
Thus, for a certain constant $c>0$:
\[\overline{d_H}(\Omega_n,\D)\geq c\frac{a_n}{n}.\]
Using this and the upper bound $\deff(\Theta_n)\leq C\frac{a_n^2}{n}$, we obtain:

\[\frac{\sigma_1(\D)-\sigma_1(\Omega_n)}{ \Big (\overline{d_H}(\Omega_n, \D)\Big )^{2-\epsilon}}\leq C a_n^\epsilon n^{1-\epsilon},\]
which is bounded for $a_n=n^{-\frac{1-\epsilon}{\epsilon}}$. For a small enough $\epsilon$, $(n^\alpha a_n)$ is bounded, so the a priori condition holds, which proves the result.
\end{proof}

\noindent
{\bf Acknowledgments.} Both authors were supported by the LabEx PERSYVAL-Lab GeoSpec (ANR-11-LABX-0025-01) and ANR SHAPO (ANR-18-CE40-0013). The authors are thankful to Iosif Polterovich for very interesting remarks and suggestions.

\bibliographystyle{plain}
\bibliography{References}

\end{document}